\documentclass{amsart}

\usepackage{amssymb,amscd,amsthm}
\usepackage{latexsym}
\usepackage{graphicx}
\usepackage{enumerate}
\usepackage{verbatim}

\date{\today}

\newcommand{\Z}{{\mathbb Z}}
\newcommand{\R}{{\mathbb R}}
\newcommand{\C}{{\mathbb C}}

\newcommand{\N}{{\mathbb N}}

\newtheorem{theorem}{Theorem}[section]

\newtheorem{lemma}[theorem]{Lemma}
\newtheorem{prop}[theorem]{Proposition}
\newtheorem{coro}[theorem]{Corollary}

\newtheorem{step}{Step}

\newcommand{\ess}{\text{\rm{ess}}}
\newcommand{\ac}{\text{\rm{ac}}}
\renewcommand{\sc}{\text{\rm{sc}}}
\newcommand{\pp}{\text{\rm{pp}}}

\newcommand{\tr}{\text{\rm{tr}}}
\newcommand{\supp}{\text{\rm{supp}}}
\newcommand{\disc}{\text{\rm{disc}}}

\newcommand{\ul}{\underline}
\newcommand{\ol}{\overline}
\def\Tr{\mathop{\rm Tr}\nolimits}
\newcommand{\SL}{\text{\rm{SL}}}

\def\B{\mathcal B}
\def\eps{\epsilon}

\def\ens{\ensuremath}                                     

\newcommand\uvm[2]{\scalebox{#1}{\ens{#2}}} 

\allowdisplaybreaks

\newcommand{\alp}{\ens{\alpha}}
\newcommand{\bet}{\ens{\beta}}

\newcommand\hsp[1]{\mbox{}\hspace{#1mm}} 
\newcommand\hspm[1]{\mbox{}\hspace{-#1mm}} 
\newcommand\vsp[1]{\par \vspace{#1mm}} 
\newcommand\vspm[1]{\par \vspace{-#1mm}} 

\newcommand{\h}{\hsp}

\numberwithin{equation}{section}

\begin{document}

\title{Ergodic Schr\"{o}dinger Operators in the Infinite Measure Setting}

\author[M.\ Boshernitzan]{Michael Boshernitzan}

\address{Department of Mathematics, Rice University, Houston, TX~77005, USA}

\email{michael@rice.edu}

\author[D.\ Damanik]{David Damanik}

\address{Department of Mathematics, Rice University, Houston, TX~77005, USA}

\email{damanik@rice.edu}

\thanks{D.D.\ was supported in part by NSF grants DMS--1361625 and DMS--1700131 and by an Alexander von Humboldt Foundation research award.}

\author[J.\ Fillman]{Jake Fillman}

\address{ Department of Mathematics, Texas State University,  San Marcos, TX 78666, USA}
	
\email{fillman@txstate.edu}
	
\thanks{J.F.\ was supported in part by an AMS-Simons travel grant, 2016-2018}

\author[M.\ Luki\'c]{Milivoje Luki\'c}

\address{Department of Mathematics, Rice University, Houston, TX~77005, USA}

\email{milivoje.lukic@rice.edu}

\thanks{M.L.\ was supported in part by NSF grant DMS--1700179.}

\begin{abstract}
We develop the basic theory of ergodic Schr\"odinger operators, which is well known for ergodic probability measures, in the case of a base dynamics on an infinite measure space. This includes the almost sure constancy of the spectrum and the spectral type, the definition and discussion of the density of states measure and the Lyapunov exponent, as well as a version of the Pastur--Ishii theorem. We also give some counterexamples that demonstrate that some results do not extend from the finite measure case to the infinite measure case. These examples are based on some constructions in infinite ergodic theory that may be of independent interest.
\end{abstract}

\maketitle
\setcounter{tocdepth}{1}
\tableofcontents

\section{Introduction}

The subject of this paper are ergodic one-dimensional discrete Schr\"odinger operators; these are self-adjoint operators $H_\omega$ on $\ell^2(\mathbb{Z})$ defined by
\begin{equation} \label{eq:schro.op}
(H_\omega u)_n = u_{n-1} + u_{n+1} + f(T^n(\omega)) u_n,
\end{equation}
where $T$ is an invertible ergodic map on a measure space $(\Omega,\mathcal B,\mu)$, $f: \Omega\to\mathbb{R}$ is bounded, and $\omega \in \Omega$.

These operators have been the subject of much research in the setting where $\mu$ is a probability measure; see, for example, \cite{CL, CFKS, D15, DF16, J07, PF} and references therein.

The subject of this paper is to explore the infinite measure setting, where $\mu(\Omega)=\infty$ and $\Omega$ is $\sigma$-finite. Infinite ergodic theory is an active area of research, but the corresponding ergodic Schr\"odinger operators have not been discussed in the literature from a global perspective.

Our work may be regarded as an initial step in the general analysis of such operators. The potentials considered here are defined by the iteration of one invertible map $T$, and hence by a $\Z$-action. This setting arises, for example, in the study \cite{G} of Schr\"odinger operators with potentials generated by almost primitive (but non-primitive) substitutions; see \cite{Y} for a discussion of the infinite invariant measures arising in that context.

Potentials generated by higher rank group actions on infinite measure spaces arise in a natural way in the analysis of quasi-periodic continuum Schr\"odinger operators via Aubry duality, compare \cite{DG}.

In the probability measure setting, the theory of ergodic Schr\"odinger operators relies on two properties of ergodic maps:
\begin{enumerate}[(i)]
\item Almost-sure constancy of invariant functions: if $f:\Omega\to\mathbb{R}$ and $f\circ T =f$ holds $\mu$-a.e., then there is a value $c\in\mathbb{R}$ such that $f = c$ $\mu$-a.e.;

\item Birkhoff's theorem: if $f \in L^1(\Omega, \mu)$ and $\mu(\Omega)=1$, then for $\mu$-a.e.\ $\omega \in \Omega$,
\begin{equation}
\lim_{n\to\infty} \frac1n\sum_{k=0}^{n-1}\limits f(T^{k}\omega) = \int f \, d\mu.
\end{equation}
\end{enumerate}
If $\mu(\Omega)=\infty$, the property (i) still holds; however, the asymptotics of Birkhoff averages become much more complicated. There is Hopf's ergodic theorem which considers ratios of Birkhoff averages for two different $L^1$ functions, but we have not found it to be of use in this setting, partly since the functions we consider are typically $L^\infty$ but not $L^1$. For our purposes, (ii) is replaced by the property
\begin{enumerate}
\item[(ii')] if $f \in L^1(\Omega, \mu)$, $\mu(\Omega)=\infty$, and $\Omega$ is $\sigma$-finite, then for $\mu$-a.e.\ $\omega \in \Omega$,
\begin{equation}\label{ml05}
\lim_{n\to\infty} \frac1n\sum_{k=0}^{n-1}\limits f(T^{k}\omega) = 0
\end{equation}
\end{enumerate}
which is an easy consequence of Hopf's ergodic theorem (see e.g.\ \cite[Exercise~2.2.1, p.~61]{A}).

To fix terminology, given a measure space $(\Omega,\mathcal B, \mu)$ and a measurable transformation $T:\Omega\to\Omega$, we will say that $T$ is \emph{invertible} if $T$ is bijective and $T^{-1}$ is measurable, \emph{measure-preserving} if $\mu(T^{-1}E) = \mu(E)$ for all $E \in \mathcal B$, \emph{ergodic} if $T^{-1}E = E$ implies $\mu(E)= 0$ or $\mu(\Omega\setminus E) =0$, and \emph{non-singular} if $\mu(E) = 0$ if and only if $\mu(T^{-1}E) = 0$. We will often assume in addition that the transformation $T$ is \emph{conservative}, which means there is no set $W \in \mathcal B$ with $\mu(W) > 0$ such that the sets $\{ T^{-n} W\}_{n=0}^\infty$ are disjoint. It is known that an invertible ergodic non-singular transformation of a non-atomic measure space is conservative, so this is a natural assumption \cite[Proposition~1.2.1]{A}. 

We will begin with a discussion of non-convergence phenomena for Birkhoff averages of $L^\infty$ functions in Sections~\ref{SBirkhoff} and \ref{sec:ex}. Specifically, Section~\ref{SBirkhoff} constructs an example with non-convergent Birkhoff averages, while Section~\ref{sec:ex} constructs an example in which the Birkhoff averages behave differently in forward and backward time.  Section~\ref{sec.basiccons} establishes basic properties of ergodic Schr\"odinger operators in the infinite measure setting. In the probability measure setting, the density of states measure and the Lyapunov exponent have a central place in the theory; their analogs are discussed in Sections~\ref{sec.dos} and \ref{sec.le}, respectively. In particular the material from Sections~\ref{SBirkhoff} and \ref{sec:ex} is used there to show that some central results known in the probability measure case do not extend to the infinite measure case.

\section{Non-Convergence of Birkhoff Averages} \label{SBirkhoff}

In what follows, let  $(\Omega,\B,\mu)$  be a  $\sigma$-finite measure space.
Denote by
\[
A_{n}(\omega,f,T)
=\frac1n\,\sum_{k=0}^{n-1}\limits f(T^{k}\omega), \quad n \in \N = \{ 1,2,\ldots \},
\]
the corresponding ergodic sums where $T\colon \Omega\to \Omega$ is a $\B$-measurable transformation
and  $f$ is a real-valued $\B$-measurable function.

In this section we prove the following result.

\begin{theorem}\label{thm:0}
Let $T\colon \Omega\to \Omega$ be a conservative, invertible, measure preserving ergodic transformation on $(\Omega,\B,\mu)$. Then there exist a  $\B$-measurable function $f \colon \Omega\to\{0,1\}$ and two strictly increasing sequences $\{p_k\}$ and $\{q_k\}$ of positive integers such that
\begin{equation}\label{eq:result1}
\lim_{k\to\infty} A_{p_k}(\omega,f,T)=1 \quad \text{ and} \quad
\lim_{k\to\infty} A_{q_k}(\omega,f,T)=0 \quad \text{{\rm (}$\mu$-a.e. $\omega \in \Omega${\rm )}.}
\end{equation}
In particular, we have
\begin{equation}
\limsup_{n\to\infty} A_{n}(\omega,f,T)=1 \h2 \text{ and } \h2
\liminf_{n\to\infty} A_{n}(\omega,f,T)=0 \qquad \text{{\rm (}$\mu$-a.e.} \ \omega\in \Omega{\rm )}.
\end{equation}
\end{theorem}

We will need the following lemma.

\begin{lemma}\label{lem:muinf1}
Let $(\Omega,\B,\mu,T)$ be as above. Let $g\in L^{1}(\Omega,\mu)$ be an integrable function
and let  $Y\in\B$ be a measurable subset of finite measure, $0<\mu(Y)<\infty$.
Then, for every $\eps>0$ and an integer $M\geq1$, there exists an integer
$N=\Phi(\eps,g,Y,M)>M$, such that
\[
\mu\big(\{\omega \in Y\,\uvm{1.2}:\,|A_{N}(\omega,g,T)|>\eps\}\big)<\eps.
\]
\end{lemma}

The statement in the above lemma follows from the relation
\[
\lim_{N\to\infty} A_{N}(\omega,g,T) = 0, \quad \text{for $\mu$-a.e.} \ \omega\in \Omega,
\]
which in turn follows from Hopf's ergodic theorem, as noted in the introduction.

For a measurable subset $Y\in\B$, denote by $\theta_{Y}$  the characteristic function of $Y$.  Clearly,  $\theta_{Y}\in L^{1}(\Omega,\mu)$ if and only if $\mu(Y)<\infty$.

\begin{proof}[Proof of Theorem \ref{thm:0}]
Select a set $Y\in\B$ of finite positive  measure, $0<\mu(Y)<\infty$. Set  $Y_{0}=Y$,
$N_{0}=1$ and construct inductively for $k\geq1$:
\begin{equation}\label{eq:nk}
N_{k}=\Phi(2^{-k},\theta_{Y_{k-1}},Y,N_{k-1}); \quad
Y_{k}=\bigcup_{j = 0}^{N_{k}-1} T^{j}(Y);\quad \text{and}
\end{equation}
\begin{equation}\label{eq:zk}
Z_k=\{\omega \in Y\,\uvm{1.2}:\, |A_{N_k}(\omega,\theta_{Y_{k-1}},T)|>2^{-k}\}.
\end{equation}

Note that we have inclusions $Y=Y_0\subset Y_1\subset Y_2\subset\cdots$, and, by the ergodicity of $T$, $\mu(\Omega\!\setminus\bigcup_{k\geq 0} Y_k)=0$. In view of Lemma \ref{lem:muinf1},
\begin{equation}\label{ineq:zk}
\mu(Z_k)<2^{-k}.
\end{equation}

Define $f\in L^{\infty}(\Omega)$ by the formula
\begin{equation}
f(\omega)=
\begin{cases}
0, &\text{ if }\ \omega \in Y;\\
(1+(-1)^{k})/2,  &\text{ if }\ \omega \in Y_k\!\setminus\! Y_{k-1}, \ k\geq1.
\end{cases}
\end{equation}

Then  we have
\[
A_{N_{k}}(\omega,\theta_{Y_{k-1}},T)\leq2^{-k}, \quad \text{for }\omega\in Y\!\setminus\! Z_{k},
\]
and
\[
\{T^j\omega :0\leq j\leq N_k-1\}\subset Y_k.
\]

If $k$ is odd, then $f(\omega)=0$ for $\omega \in Y_k\!\setminus\! Y_{k-1}$ and hence
\begin{equation}\label{ineq:odd}
0\leq A_{N_k}(\omega,f,T)\leq A_{N_k}(\omega,\theta_{Y_{k-1}},T)\leq2^{-k},\quad
\text{for }\omega \in Y\!\setminus\! Z_{k},
\end{equation}
since $f(\omega)\leq \theta_{Y_{k-1}}(\omega)$  for $\omega \in Y_{k}$.

Similarly, if $k \geq 2$ is even, then $f(\omega)=1$ for $\omega \in Y_k\!\setminus\! Y_{k-1}$
and hence
\[
0\leq A_{N_k}(\omega,1-f,T)
\leq A_{N_k}(\omega,\theta_{Y_{k-1}},T)\leq2^{-k},\quad
\text{for }\omega\in Y\!\setminus\! Z_{k},
\]
since $1-f(\omega)\leq \theta_{Y_{k-1}}(\omega)$  for $\omega\in Y_{k}$. It follows that
\begin{equation}\label{ineq:even}
1
\geq A_{N_k}(\omega,f,T)
\geq 1-2^{-k},\quad
\text{for }\omega\in Y\!\setminus\! Z_{k}.
\end{equation}

By the Borel-Cantelli lemma, the inequality \eqref{ineq:zk} implies that $\mu(W)=0$  where
\[
W=\limsup Z_k=\bigcap_{n\geq1}\Big(\bigcup_{k=n}^\infty Z_k \Big).
\]
In view of the inequalities \eqref{ineq:odd} and \eqref{ineq:even}, we obtain
\begin{equation}\label{eq:extend}
\lim_{k\to\infty} A_{N_{2k+1}}(\omega,f,T)=0, \quad \lim_{k\to\infty} A_{N_{2k}}(\omega,f,T)=1,
\end{equation}
for all $\omega\in Y\!\setminus\! W$.

Since $\mu(Y\!\setminus\! W)=\mu(Y)>0$ and $T$ is ergodic, the relations \eqref{eq:extend} extend to $\mu$-a.e.\ $\omega \in \Omega$. One takes $p_k=N_{2k}$ and $q_k=N_{2k+1}$ to complete the proof of Theorem \ref{thm:0}.
\end{proof}

\section{Different Behaviors for Birkhoff Averages in Forward and Backward Time}\label{sec:ex}

We describe an example of a conservative, invertible, measure preserving ergodic transformation $(\Omega,\B,\mu,T)$ and a $\B$-measurable function $f\colon \Omega\to \{0,1\}$ such that
\begin{equation}\label{eq:an1}
\limsup_{n\to+\infty} A_{n}(\omega,f,T)=1,
\end{equation}
while
\begin{equation}\label{eq:an0}
\lim_{n\to+\infty} A_{n}(\omega,f,T^{-1})=0,
\end{equation}
(both) for all $\omega \in \Omega$.

We wish to thank Benjy Weiss for referring us to an old paper by Dowker and Erd\H{o}s \cite{DE}, where similar examples are described (in a somewhat different setting).

Set $K=(0,1]=\R/\Z$ to be the (left-open and right-closed) unit interval naturally identified with the circle, let $\alp$ be any badly approximable irrational number (i.e.~one with bounded partial quotients), and denote by $R\colon K\to K$ the $\alp$-rotation on $K$ (determined by the identity $R(x)=x+\alp \pmod1$).

One proceeds by setting
\begin{equation}\label{eq:X}
\Omega=\big\{(u,m)\in K\times \Z:  1\leq m\leq h(u)\big\}\subset\R^2,
\end{equation}
where\, $h\colon K\to\R$ is defined by the formula
\[
h(u):=2^{\,2^{^{\frac{u+1}u}}}, 
\quad u\in K.
\]

Next, one defines the function $f\colon \Omega\to \R$\, by the formula
\begin{equation}\label{eq:f}
f(u,m)=
\begin{cases}
1, &\text{ if } \ m\leq \sqrt{h(u)}=2^{\,2^{^{\frac1u}}}, \\ 
0, &\text{ otherwise,}
\end{cases}
\end{equation}
and the map  $T\colon \Omega\to \Omega$ by the formula
\begin{equation}\label{eq:T}
T(u,m)=\begin{cases}
(u,m+1),& \text{ if }\ (u,m+1)\in \Omega;\\
(R(u), 1),& \text{ otherwise}.
\end{cases}
\end{equation}

In other words, $(\Omega,T)$ is the suspension $\Z$-flow over the rotation $(K,R)$ with the delay function $\big[h(u)\big]$. Note that $h$ is strictly decreasing and
\[
\min_{u\in K}\,h(u)=h(1)=16.
\]
Finally, consider the product of Lebesgue measure on $\R$ and counting measure on $\Z$, and let $\mu$ denote the restriction of this measure to $\Omega$. Since  $\mu(\Omega)=\int_0^1 [h(t)]\,dt=\infty$, it follows that $T$ is a  conservative, invertible, measure preserving ergodic transformation on the infinite measure space $(\Omega,\mathcal B, \mu)$ (see e.g.\ \cite[Section~1.2]{DE}).

 For these choices of $(\Omega, T)$ and $f$, we shall validate both relations \eqref{eq:an1} and \eqref{eq:an0} (in Subsections \ref{sec:proof1} and~\ref{sec:proof2}, respectively).

\subsection{Some notation}\label{ss.notation}

Let $\omega_0=(u_0,m_0)\in \Omega$ be fixed. Set
\begin{equation}\label{eq:xk}
\omega_k=(u_k,m_k)=T^k \omega_0, \quad \forall k\in\Z.
\end{equation}
It is enough to prove \eqref{eq:an1} and \eqref{eq:an0} for $\omega = \omega_0$,  under the added assumption that $m_0=1$ (because $(u,1)=T^{-(m-1)}(u,m)$  lies in the  $T$-orbit of every $(u,m)\in \Omega$).

Since the set
\[
\{k\in\Z: u_{k-1}\neq u_k\}\subset\Z
\]
contains $0$ and is unbounded from both below and above, it could be uniquely arranged into an infinite two-sided increasing sequence of integers ${\bf t}=(t_n)_{n\in\Z}$, with $t_0=0$:
\[
\cdots < t_{- 2 } < t_{-1} < t_0 = 0 < t_1 < t_2 < \cdots.
\]

The set of integers $\Z$ is partitioned into finite subsets
\begin{equation}\label{eq:wk}
V_k=\big[t_k,t_{k+1}\big)\cap\Z,\qquad|V_k|=t_{k+1}-t_k=\big[h(\bet_k)\big]
\end{equation}
where
\begin{equation}\label{eq:bk}
\beta_k = u_{t_k} = R^k(u_0).
\end{equation}
In fact, we have
\begin{equation}\label{eq:un}
u_n=u_{t_k}=\beta_k=R^k(u_0),\quad \forall n\in V_k.
\end{equation}

Next we set
\begin{equation}\label{eq:pk}
p_k=\big[h(\bet_k)\big]=t_{k+1}-t_k=|V_k|, \quad q_k=\big[\sqrt{p_k}\big], \qquad \forall k\in\Z.
\end{equation}
Then the following $p_k$-tuples of $0$'s and $1$'s coincide:
\begin{equation}\label{eq:block1}
\big(f(\omega_n)\big)_{n=t_k}^{t_{k+1}-1}=\big((1)_{q_k},(0)_{p_k-q_k}\big) :=
    (\,\underbrace{1,\ldots,1}_{q_k\text{ times}},\underbrace{0,\ldots,0}_{\ p_k-q_k
    \text{ times}}\!), \quad \forall k\in\Z.
\end{equation}

Since $t_0=0$, the second equality in \eqref{eq:wk} implies that
\begin{subequations}\label{eq:tn}
\begin{align}
t_n\ &=\sum_{k=1}^n (t_k-t_{k-1})=\sum_{k=0}^{n-1}\,p_k, \quad \forall n\geq0, \quad \text{ and}\label{eq:tn1}\\
t_{-n}&=-\hspm2\sum_{k=-n+1}^0 (t_k-t_{k-1})
=-\hspm2\sum_{k=-n}^{-1} p_k=\\&=-\sum_{k=1}^{n} p_{-k}, \quad \forall n\in\N\notag.
\end{align}
\end{subequations}

\subsection{Proof of \eqref{eq:an1}}\label{sec:proof1}

Since $\alpha$ is irrational, the sequence  $(\bet_n)_{n=0}^\infty$ is dense in $K=(0,1]$ (see \eqref{eq:bk}), so it achieves its minimum infinitely many times. That is, the set
\[
S=\{n\geq2: \bet_n=\mu_n\}=\{n\geq2: \mu_n<\mu_{n-1}\}
\]
is infinite where
\[
\mu_n=\min_{0\leq k\leq n} \bet_k,\quad\forall n\in \N.
\]

Since $\alp$ is a badly approximable irrational, for all integers $n\geq1$, the $n+1$ points $\bet_k,\,k=0,1,2,\ldots,n$, are all different and they partition \mbox{$K=(0,1]$} (viewed as the unit circle) into subintervals of proportional lengths. By ``proportional lengths'' we mean that the ratio of the lengths of any two such subintervals is bounded by a constant $c_1=c_1(\alp)>1$, which is independent of $n$. This follows, for example, from the three distance theorem, compare \cite{AB}.

It follows that $\mu_{n}\leq \frac{c_1}{n+1}<\frac {c_1}n$, and that, for all $0\leq k\leq n$ such that $\bet_k\neq \mu_n$, we have  $\frac{\mu_n}{\bet_k-\mu_n}\leq c_1$ and hence  $\frac{\bet_k}{\mu_n}\geq \frac{1+c_1}{c_1}$. Set $c_2=\frac{1+c_1}{c_1}>1$ to be a new constant.

Then, for the above constants  $c_1,c_2 > 1$, we have
\begin{equation}\label{ineqs}
1<c_2\leq \frac{\bet_k}{\bet_n}; \quad \bet_n<\frac {c_1}n, \qquad
   \forall n\in S\ \text{\ {\small and }\, } 0\leq k\leq n-1,
\end{equation}
because $\bet_n = \mu_n$, for all $n\in S$.

The following estimate for all $n\in S$ follows from \eqref{eq:tn1}, \eqref{ineqs} and the fact that $h$ is decreasing:
\begin{align*}
t_n=\sum_{k=0}^{n-1}\,[h(\bet_k)]
\leq \sum_{k=0}^{n-1}\,h(c_2\bet_n)
=n\,h(c_2\bet_n)
<\frac{c_1}{\bet_n}\,h(c_2\bet_n)=\frac{c_1}{\mu_n}\,h(c_2\mu_n).
\end{align*}
Observe that $q_n=\Big[\sqrt{h(\bet_n)}\,\Big]=\Big[\sqrt{h(\mu_n)}\,\Big]$\, for $n\in S$
(see \eqref{eq:pk}) and that
\begin{equation}\label{eq:tgn}
\lim_{\substack{n\in S\\n\to\infty}}\,\frac{t_n}{q_n}=0,
\end{equation}
because $\displaystyle\lim_{n\to\infty} \mu_n=0$\, and $\displaystyle\lim_{x\to 0+} \uvm{1.1}{\tfrac{c_1\,h(c_2x)}{x\sqrt{h(x)}}}=0$ (here the assumption $c_2>1$ is used).

In order to prove \eqref{eq:an1}, it is enough to show that
\begin{equation}\label{eq:an2}
\lim_{\substack{\\n\in S\\n\to+\infty}}  A_{t_n+q_n}(\omega_0,f,T)=1
\end{equation}
because $A_n(\omega_0,f,T)\leq1$, $\forall n\in\N$.
We start by splitting the sum
\begin{align*}
A_{t_n+q_n}(\omega_0,f,T)
&=\frac1{t_n+q_n}\sum_{k=0}^{t_n+q_n-1}f(T^k\omega_0)=\\
   &=\frac1{t_n+q_n}\,\Bigg(\bigg(\sum_{k=0}^{t_n-1}+\sum_{k=t_n}^{t_n+q_n-1}\bigg)\,f(T^k\omega_0)\Bigg)
   =\frac{S_1+S_2}{t_n+q_n},
\end{align*}
where
\[
0\leq S_1=\sum_{k=0}^{t_n-1}\,f(T^k\omega_0)\leq t_n
\]
and, for $n\in S$,
\[
S_2
=\sum_{k=t_n}^{t_n+q_n-1}\,f(T^k\omega_0)
=\sum_{k=t_n}^{t_n+q_n-1}\,1=q_n
\]
(in view of the definition of $f$ and since $q_n=\Big[\sqrt{h(\bet_n)}\Big]$, for $n\in S$). It follows from \eqref{eq:tgn} that
\[
0\leq\limsup_{\substack{\\n\in S\\n\to+\infty}}\,\frac{S_1}{t_n+q_n}\leq
     \limsup_{\substack{\\n\in S\\n\to+\infty}}\,\frac{t_n}{t_n+q_n}=0
\]
and that
\[
\lim_{\substack{\\n\in S\\n\to+\infty}}\frac{S_2}{t_n+q_n}=\lim_{\substack{\\n\in S\\n\to+\infty}}\frac{q_n}{t_n+q_n}=1,
\]
whence \eqref{eq:an2} follows. This completes the proof of \eqref{eq:an1}.

\subsection{More notation and estimates}

We assume the conventions and notation introduced above, in particular \eqref{eq:xk}, \eqref{eq:wk} and \eqref{eq:tn}. We also set new sequences
\begin{equation*}
V'_k=-V_{-k},\ t'_{k}=-t_{-k}, \ p'_k=p_{-k},\ q'_k=q_{-k}, \quad\forall k\in\Z.
\end{equation*}

Then we have
\begin{equation*}
\ldots<t'_{-2}<t'_{-1}<t'_0=0<t'_1<t'_2<\ldots,
\end{equation*}
and
\begin{equation*}
V'_k=(t'_{k-1},t'_k]\cap\Z=\big[t'_{k-1}+1,t'_k\big]\cap\Z.
\end{equation*}

Thus (see \eqref{eq:pk})
\begin{equation}\label{eq:vksize}
|V'_k|=|V_{-k}|=\big[h(\beta_{-k})\big]=p_{-k}=p'_{k}, \quad \forall k\in\Z.
\end{equation}
\vsp1

\begin{prop}\label{prop:block}
For all $k\in\Z$,  we have
\[
(f(\omega_{-t'_k}),f(\omega_{-(t'_k-1)}),\ldots,f(\omega_{-(t'_{k-1}+1)}))=(1)_{q'_k}(0)_{p'_k-q'_k}.
\]
\end{prop}

\begin{proof}
The verification is straightforward:
\begin{align*}
&(f(\omega_{-t'_k}),f(\omega_{-(t'_k-1)}),\ldots,f(\omega_{-(t'_{k-1}+1)}))
=\\
&=(f(\omega_{t_{-k}}),f(\omega_{t_{-k}+1}),\ldots,f(\omega_{t_{-(k-1)}-1}))=\\
&=(f(\omega_{n}))_{n=t_{-k}}^{t_{-(k-1)}-1}=
(1)_{q'_k}(0)_{p'_k-q'_k} \qquad (\text{here }\eqref{eq:block1}\text{ is used)}.
\end{align*}
\end{proof}

\begin{prop}\label{prop:sumblock}
For all $k\in\Z$,  we have
\[
\sum_{n\in V'_k} f(\omega_{-n})=q'_k.
\]
\end{prop}

\begin{proof}
This follows from Proposition \ref{prop:block}.
\end{proof}

\subsection{Proof of \eqref{eq:an0}}\label{sec:proof2}

Denote
\begin{equation}\label{eq:andef}
a_n=\frac1n\ \sum_{k=1}^n f(\omega_{-k}),\quad n\geq1.
\end{equation}
where  $\omega_k=T^k\omega_0=(u_k,m_k)$ for $k\in\Z$.

In order to prove \eqref{eq:an0}, it is enough to show that
\begin{equation}\label{eq:anlim}
\lim_{n\to+\infty} a_n=0.
\end{equation}

\begin{lemma}\label{eq:anbetween}
For integers $n\geq t'_1$, we have $0<a_n<1$.
\end{lemma}

\begin{proof}
This follows from \eqref{eq:andef} because $f$ is $\{0,1\}$-valued and $f(\omega_{-1})=0$ while $f(\omega_{-t'_1})=1$ (in view of Proposition~\ref{prop:block} with $k=1$).
\end{proof}

\begin{lemma}\label{lem:ineqan}
Assume that for some $k\geq2$ we have $n\in V'_k=(t'_{k-1},t'_k]\cap\Z$. Then
\[
a_n\leq\max(a_{t'_{k-1}},a_{t'_k}).
\]
\end{lemma}

\begin{proof}
The identity $(n+1)(a_{n+1}-a_n)=f(\omega_{-(n+1)})-a_{n}$ (which holds for all $n\geq1$) and Lemma~\ref{eq:anbetween}
imply the inequalities
\[
\begin{cases}
a_{n+1}>a_n, & \text{if }\, f(\omega_{-(n+1)})=1,\\
a_{n+1}<a_n, & \text{if }\, f(\omega_{-(n+1)})=0,
\end{cases}
\qquad \text{ (for }n\geq t'_1).
\]

By Proposition \ref{prop:block}, we obtain the inequalities
\begin{align*}
a_{t'_k}&>a_{t'_k-1}>a_{t'_k-2}>\ldots >a_{t'_k-q'_k}=\\
&=a_{t'_k-q'_k}<a_{t'_k-q'_k-1}<\ldots<a_{t'_{k-1}+1}<a_{t'_{k-1}},
\end{align*}
whence the claim of Lemma \ref{lem:ineqan} follows.
\end{proof}

We conclude from Lemma \ref{lem:ineqan} that in order to establish the limit \eqref{eq:anlim}, it suffices to do it only over the subsequence $(t'_k)$, i.\,e.  to prove that
\begin{equation}\label{eq:avk}
\lim_{k\to+\infty} a_{t'_k}=0.
\end{equation}
\vspm3

We have
$\displaystyle a_{t'_k}=\frac1{t'_k}\sum_{i=1}^{t'_{k}}f(\omega_{-i})$.
Since $\big[1,t'_{k}\big]\cap\N$ can be partitioned into the disjoint union
\[
\big[1,t'_{k}\big]\cap\N=\cup_{j=1}^{k} V'_j,
\]
and since $|V'_j|=p'_j$  and  $\displaystyle\sum_{n\in V'_j}\,f(\omega_{-n})=q'_j$ (see \eqref{eq:vksize} and Proposition \ref{prop:sumblock}, respectively), we obtain
\[
a_{t'_k}=\frac{\sum_{j=1}^{k}q'_j}{\sum_{j=1}^{k}p'_j}.
\]
(Recall that $p'_j=p_{-j}=[h(R^{-j}(u_0))]\geq16$ and $q'_j=\left[\sqrt{p'_j}\ \right]\leq \sqrt{p'_j}\geq4$, see \eqref{eq:vksize}, \eqref{eq:un} and~\eqref{eq:pk}).

By the Cauchy-Schwarz (or Jensen) inequality, we have
\[
\Big(\sum_{j=1}^{k}q'_j\Big)^2\leq k\sum_{j=1}^k (q'_j)^2\leq k\sum_{j=1}^k p'_j
\]
whence
\[
(a_{t'_k})^2=\left(\frac{\sum_{j=1}^{k}q'_j}{\sum_{j=1}^{k}p'_j}\right)^2\leq \frac k{\sum_{j=1}^{k}p'_j}.
\]

It remains to prove that
\[
\lim_{k\to+\infty}\frac{\sum_{j=1}^k p'_j}k=+\infty
\]
because then \eqref{eq:avk} and hence \eqref{eq:an0} follow. But
\[
\frac{\sum_{j=1}^k p'_j}k=\frac{\sum_{j=1}^k \big[h(R^{-j}(u_0))\big]}k
\]
is just the $k$-th ergodic average of the positive function $\big[h(u)\big]$ for the irrational rotation $R^{-1}$ evaluated at $u_0$. Therefore
\[
\lim_{k \to +\infty}\frac{\sum_{j=1}^k p'_j}k=\int_{0}^{1} \big[h(u)\big]\,du=+\infty,
\]
completing the proof of \eqref{eq:an0}.

\section{Basic Consequences of Ergodicity}\label{sec.basiccons}

Throughout the remainder of the paper, let $(\Omega, \mathcal{B}, \mu)$ denote a $\sigma$-finite measure space with $\mu(\Omega) = \infty$, $T:\Omega \to \Omega$ a conservative, invertible, ergodic, measure-preserving transformation, and $f:\Omega\to\R$  bounded and measurable. For each $\omega \in \Omega$,  $H_\omega = \Delta + V_\omega $ is defined by \eqref{eq:schro.op}. Throughout, $S:\ell^2(\mathbb{Z}) \to \ell ^2 (\mathbb{Z})$ will denote the left shift $S: \delta_n \mapsto \delta_{n-1}$.

\begin{lemma} \label{l:regfcts:weak.msrbl}
If $h:\R \to \mathbb{C}$ is a locally bounded Borel function, then the family
$\{h(H_{\omega})\}_{\omega \in \Omega}$ is weakly measurable in the sense that
$$
I_h^{\phi, \psi}(\omega) := \langle \phi, h(H_{\omega})\psi \rangle
$$
defines a measurable function $\Omega \to \mathbb{C}$ for all $\phi, \psi \in \ell^2(\mathbb{Z})$.
\end{lemma}

\begin{proof}
Let $K = [-2-\|f\|_\infty,2+\|f\|_\infty]$. Since $\sigma(H_\omega) \subseteq K$ for $\mu$-a.e.\ $\omega$, it suffices to prove the theorem for bounded Borel functions $K \to \C$. Let $\mathcal{A}$ denote the set of bounded Borel functions $h:K\to\mathbb{C}$ such that $I_h^{\phi,\psi}(\omega)$ is a measurable function $\Omega \to \mathbb{C}$ for all $\phi, \psi \in \ell^2(\mathbb{Z})$. It is clear that $\mathcal{A}$ is a vector subspace of all bounded Borel functions and that it contains the constant function $h\equiv 1$. We will now prove some additional properties of $\mathcal{A}$.

\begin{step} \label{step:id} \normalfont
\textbf{\boldmath $\mathcal{A}$  contains the function $h(x) = x$.} Measurability of $f, T$, and $T^{-1}$ imply that
$$
I_h^{\phi,\psi}(\omega) = \langle \phi, \Delta \psi \rangle + \sum_{k\in\mathbb{Z}} \overline{\phi_k} \psi_k f(T^k\omega)
$$
is a measurable function of $\omega$.
\end{step}

\begin{step} \label{step:prod} \normalfont
\textbf{\boldmath If $g, h \in \mathcal{A}$, then $gh \in \mathcal{A}$.} This follows from
\[
\langle \phi, g(H_\omega) h(H_\omega) \psi \rangle = \sum_{n\in \mathbb{Z}}
\langle \phi, g(H_\omega) \delta_n \rangle \langle \delta_n, h(H_\omega) \psi
\rangle
\]
because products and pointwise limits of measurable functions are measurable.
\end{step}

\begin{step} \label{step:conv} \normalfont
\textbf{\boldmath If $g_n \in \mathcal{A}$ for all $n\in \mathbb{N}$, $g_n$ are uniformly bounded, and $g_n \to g$ pointwise, then $g \in \mathcal{A}$.} If $g_n$ are uniformly bounded and converge to $g$ pointwise, then $g_n(H_\omega)\stackrel{s}{\longrightarrow} g(H_\omega)$ by \cite[Theorem 3.1]{T}, so $I_{g_n}^{\phi,\psi}$ converge pointwise to $I_g^{\phi,\psi}$. As the pointwise limit of measurable functions, $I_g^{\phi,\psi}$ is measurable.
\end{step}

\begin{step} \label{step:CK} \normalfont
\textbf{\boldmath $C(K) \subset \mathcal{A}$.} Since $\mathcal{A}$ is an algebra and contains the functions $1$ and $x$, it contains all polynomials. Since it is closed under uniform limits, by Weierstrass' theorem $\mathcal{A}$ contains all continuous functions.
\end{step}

\begin{step} \label{step:char2} \normalfont
\textbf{\boldmath The set $\mathcal{E}$ of Borel sets $B \subset K$ such that $\chi_B \in \mathcal{A}$ is a $\sigma$-algebra.} It is clear that $\emptyset \in\mathcal{E}$. Since $\chi_B \in \mathcal{A}$ implies $\chi_{K\setminus B} = 1 - \chi_B \in \mathcal{A}$, $\mathcal{E}$ is closed under taking complements. If $B_1, B_2 \in \mathcal{E}$, then $B_1 \cap B_2 \in \mathcal{E}$ because $\chi_{B_1 \cap B_2} = \chi_{B_1} \chi_{B_2}$. Thus, $\mathcal{E}$ is closed under finite intersections and therefore finite unions. Finally, $\chi_{\cup_{n=1}^\infty B_n} = \lim_{N\to\infty} \chi_{\cup_{n=1}^N B_n}$ implies that $\mathcal{E}$ is closed under countable unions.
\end{step}

\begin{step} \label{step:char1} \normalfont
\textbf{\boldmath $\chi_B \in \mathcal{A}$ for all Borel sets $B \subset K$.} For any closed $F \subset K$, the characteristic function $\chi_F$ is the limit of continuous functions $\max(1 - n \mathrm{dist}(x,F), 0)$ as $n\to\infty$, by Step~\ref{step:CK}, $\chi_F \in \mathcal{A}$. Thus, the $\sigma$-algebra $\mathcal{E}$ contains all closed sets, so it contains $\mathcal B$, the Borel $\sigma$-algebra.
\end{step}

\begin{step} \label{step:char1} \normalfont
\textbf{\boldmath $h \in \mathcal{A}$ for all bounded Borel functions $h:K\to\mathbb{C}$.} The set $\mathcal{A}$ contains all simple functions as linear combinations of characteristic functions. Since every bounded Borel function can be uniformly approximated by simple functions, $\mathcal{A}$ contains all bounded Borel functions.
\end{step}
\end{proof}

\begin{lemma} \label{l:weak.msrbl.proj}
Let $(\Omega, \mathcal{B}, \mu)$ be as above, and suppose that $T$ is ergodic, invertible, and conservative.  Suppose further that $\{P_{\omega}\}_{\omega \in \Omega}$ is a weakly
measurable family of orthogonal projections such that $P_{T\omega} = SP_{\omega}S^*$.  Then $\tr(P_\omega) = \dim(\textup{range}(P_{\omega}))$ is $\mu$-almost surely constant -- moreover, the $\mu$ almost sure value of $\tr(P_\omega)$ must be either 0 or $\infty$.
\end{lemma}

\begin{proof}
By weak measurability, $\langle \delta_n, P_\omega \delta_n \rangle $ is a measurable function of $\omega \in \Omega$ for each $n \in \Z$, so $Q(\omega) := \mbox{tr}(P_{\omega}) $ is a measurable
function of $\omega$.  Moreover, $Q$ is $T$-invariant, since
\begin{align*}
Q(T\omega) & = \sum_{n\in \mathbb{Z}} \langle \delta_n , P_{T\omega}\delta_n\rangle \\
& = \sum_{n\in \mathbb{Z}} \langle \delta_n , SP_{\omega}S^*\delta_n\rangle \\
& = \sum_{n\in \mathbb{Z}} \langle \delta_{n+1} , P_{\omega}\delta_{n+1}\rangle \\
& = Q(\omega).
\end{align*}
Thus, by ergodicity of $T$, there exists some $c \in [0, \infty]$ such that $Q(\omega) = c$ for $\mu$-almost every $\omega$.  To conclude the proof, we note that $Q \geq 0$, so it suffices to show that $c>0$ implies $c = \infty$.  To that end, assume that $c>0$ and consider $f:\Omega \to \R$ defined by
\begin{equation}
f(\omega)
=
\langle \delta_0, P_\omega \delta_0 \rangle.
\end{equation}
Notice that $f$ is nonnegative and that
\begin{equation}
f(T^n \omega)
=
\langle \delta_0, P_{T^n\omega} \delta_0 \rangle
=
\left\langle \delta_0, S^n P_\omega \left(S^*\right)^n \delta_0 \right\rangle
=
\langle \delta_n, P_\omega \delta_n \rangle,
\end{equation}
so $f$ need not be $T$-invariant.  However, $f$ cannot vanish almost everywhere, for, if $f(\omega) = 0$ for $\mu$-almost every $\omega$, then by taking a countable intersection of sets of full $\mu$-measure, we would have a full-measure set of $\omega$ with $f(T^n \omega) = 0$ for all $n \in \Z$ and hence
$$
Q(\omega)
=
\sum_{n \in \Z} f(T^n\omega) = 0
$$
for all such $\omega$, i.e. $c=0$.  In particular, since $f$ does not vanish almost everywhere, we may choose $\delta > 0$ and $\Omega_1 \subseteq \Omega$ with $\mu(\Omega_1) > 0$ and $f(\omega) \geq 2\delta$ for all $\omega \in \Omega_1$.  By removing a set of $\mu$-measure zero from $\Omega_1$, we may assume without loss that $Q(\omega) = c$ for all $\omega \in \Omega_1$ as well.  By Poincar\'e recurrence (\cite[Theorem~1.1.5]{A}), we may throw out yet another set of measure zero to get
$$
\liminf_{n \to \infty} | f(\omega) - f(T^n\omega) | = 0
$$
for all $\omega \in \Omega_1$.  Thus, to every $\omega \in \Omega_1$ there corresponds a sequence $n_j = n_j(\omega) \to \infty $ with
\begin{equation}
f(T^{n_j}\omega) \geq \delta
\end{equation}
for all $j$.  Evidently then,
\begin{equation}
Q(\omega)
=
\sum_{n \in \Z} f(T^n \omega) = \infty
\end{equation}
for all $\omega \in \Omega_1$. Therefore, $c = \infty$, as claimed.
\end{proof}

Notice that Lemma \ref{l:weak.msrbl.proj} need not hold if $T$ is dissipative.  Indeed, consider $\Omega = \Z$ endowed with counting measure and $T:n \mapsto n-1$.  For $n \in \Z$, let $P_n$ denote orthogonal projection onto the one-dimensional subspace spanned by $\delta_n$.  It is easy to see that
$$
SP_nS^* = P_{n-1} = P_{Tn},
$$
but $\tr(P_n) = 1$ for all $n \in \Z$.

\begin{theorem} \label{t:as.spectrum}
There exists a compact set $\Sigma \subseteq \mathbb{R}$ such that $\sigma(H_{\omega}) = \Sigma$ for $\mu$-almost every $\omega \in \Omega$.
\end{theorem}

\begin{proof}
For $-\infty < p < q < \infty$, Lemma \ref{l:regfcts:weak.msrbl} implies that $ (\chi_{(p,q)}(H_\omega))_{\omega \in \Omega} $ is a weakly measurable family of projections.  Let $d_{p,q}$ denote the almost sure value of $\tr(\chi_{(p,q)}(H_{\omega}))$, which
is either $0$ or $\infty$ by Lemma \ref{l:weak.msrbl.proj}.  Next, let $\Omega_{p,q}$ denote the (full measure) set of $\omega\in\Omega$ for which
$\tr(\chi_{(p,q)}(H_{\omega})) = d_{p,q}$ and define
$$
\Omega_0 := \bigcap_{p<q, p,q \in \mathbb{Q}} \Omega_{p,q},
$$
which is a set of full measure.  Now, for all $\omega, \tilde{\omega}\in\Omega_0$, we claim that
$\sigma(H_{\omega})=\sigma(H_{\tilde{\omega}})$.  To see this, assume $E \in \mathbb{R} \setminus \sigma(H_{\omega})$.  Then we can choose
$p < q$ rational with $E \in (p,q) \subseteq \mathbb{R} \setminus \sigma(H_{\omega}) $.  One then has
$$
0 = \mbox{tr}(\chi_{(p,q)}(H_{\omega})) = d_{p,q} = \mbox{tr}(\chi_{(p,q)}(H_{\tilde{\omega}})),
$$
which implies that $E \in \mathbb{R} \setminus \sigma(H_{\tilde{\omega}})$.  By symmetry, we are done.
\end{proof}

\begin{coro}
For all $E \in \mathbb{R}$, one has $\mu(\{\omega : E\in \sigma_{\mathrm p}(H_{\omega})\})=0$.
\end{coro}

\begin{proof}
Since the space of sequences $u:\Z \to \C$ with
$$
u_{n-1} + u_{n+1} + V_\omega(n) u_n
$$
is two-dimensional, it follows that
$$
\tr(\chi_{\{E\}}(H_{\omega}))
=
\dim(\ker(H_\omega - E))
\leq
2
<
\infty
$$
for all $\omega$.  By Lemma \ref{l:weak.msrbl.proj}, the almost sure value of
$\tr(\chi_{\{E\}}(H_{\omega}))$ must be zero, i.e. $E$ is $\mu$-almost surely not an eigenvalue of $H_{\omega}$.
\end{proof}

\begin{coro}
One has $\mu(\{\omega : \sigma_{\disc}(H_{\omega}) \neq \emptyset \})=0$.
\end{coro}

\begin{proof}
Suppose $\sigma_{\disc}(H_{\omega}) \neq \emptyset$. Given $E \in
\sigma_{\disc}(H_{\omega})$, there exist rational numbers $p < q$ such that $ (p,q) \cap \sigma(H_\omega) = \{E\} $, so
$$
\tr(\chi_{(p,q)}(H_{\omega}))
=
\dim(\ker(H_\omega - E))
=
1.
$$
In particular, following the notation in the proof of Theorem~\ref{t:as.spectrum}, $\omega \notin \Omega_{p,q} $, so $ \omega \notin \Omega_0 $.
\end{proof}

The arguments given in \cite{CFKS} and \cite{T00} generalize without modification to establish $\mu$-almost everywhere constancy of the spectral decomposition into absolutely continuous, singular continuous and pure point parts.

\begin{theorem}
There exist compact sets $\Sigma_{\ac},\Sigma_{\sc},\Sigma_{\pp} \subseteq \R$ so that for $\mu$-almost every $\omega \in \Omega$, one has $\sigma_\bullet(H_{\omega}) = \Sigma_\bullet $ for $\bullet \in \{ \ac, \sc, \pp \}$.

\end{theorem}

\begin{proof}
Let $ \mathcal{P}^{\ac}_\omega, \mathcal{P}^{\sc}_\omega $, and $ \mathcal{P}^{\pp}_\omega $ denote projection onto the absolutely continuous, singular continuous, and pure point subspaces corresponding to $H_\omega$, respectively.  By following the argument in Theorem \ref{t:as.spectrum}, it clearly suffices to prove weak measurability of these three families of projections.  If $\mathcal{P}^{\mathrm c}_\omega$ denotes projection onto the continuous subspace of $H_\omega$, we have
\begin{equation} \label{eq:contproj}
\langle \phi, \mathcal{P}^{\mathrm c}_\omega \psi \rangle
=
\lim_{N \to \infty} \lim_{T \to \infty} \frac{1}{T} \int_0^T \! \left\langle \phi, e^{itH_\omega}(1-\chi_{[-N,N]}) e^{-itH_\omega} \psi \right\rangle,
\end{equation}
by \cite[Equation~(5.20)]{T00}.  If $\mathcal{P}^{\mathrm s}_\omega$ denotes projection onto the singular subspace of $H_\omega$, then
\begin{equation} \label{eq:singproj}
\langle \phi, \mathcal{P}^{\mathrm s}_\omega \psi \rangle
=
\inf_{\delta > 0} \sup_{\mathrel{\substack{I \in \mathcal{I} \\ |I| < \delta}}} \langle \phi, \chi_I(H_\omega) \psi \rangle ,
\end{equation}
where $\mathcal I$ denotes the collection of intervals in $\R$ with rational endpoints. This follows from \cite[Lemma~B.6]{T00}.
Since $ \mathcal{P}^{\mathrm c}_\omega =  \mathcal{P}^{\ac}_\omega + \mathcal{P}^{\sc}_\omega$ and $ \mathcal{P}^{\mathrm s}_\omega = \mathcal{P}^{\sc}_\omega + \mathcal{P}^{\pp}_\omega $, we are done.
\end{proof}

\section{The Density of States}\label{sec.dos}

In this section, we explore possible notions of the density of states for an ergodic family $(H_\omega)_{\omega \in \Omega}$.

\subsection{Ergodic Averages of Spectral Measures}
When the underlying measure $\mu$ is a probability measure, one can view the density of states as the $\mu$-average of the $\delta_0$ spectral measures of the family $H_\omega$, i.e.
\begin{equation} \label{eq:prob.ids}
\int_{\R} \! g(E) \, dk(E)
=
\int_{\Omega} \! \langle \delta_0, g(H_\omega) \delta_0 \rangle\, d\mu(\omega).
\end{equation}
By Cauchy-Schwarz, the integrand on the right hand side is bounded and hence is $L^1$ with respect to $\mu$.  In the case when $\mu$ is an infinite measure and $\Omega$ is $\sigma$-finite, we clearly need to treat convergence issues with more care.  One way to work around this is to exhaust $\Omega$ by subsets of finite $\mu$-measure and then attempt to understand the natural restrictions of \eqref{eq:prob.ids} to these subsets.  More precisely, let $\mathcal{F}$ denote the collection of measurable subsets $F \subseteq \Omega$ having finite $\mu$-measure.  For each $F \in \mathcal{F}$, define a probability measure $d k^{F}$ by
\begin{equation} \label{eq:ids.def}
\int \! g(E) \, d k^{F}(E)
=
\frac{1}{\mu(F)} \int_{F} \! \langle \delta_0 , g(H_{\omega}) \delta_0\rangle \, d \mu(\omega)
\end{equation}
for each continuous function $g$ having compact support.  Obviously, $d k^{F}$ is absolutely continuous with respect to $d k^{F'}$ whenever $F \subseteq F'$, since
\begin{equation}
k^{F'}(B)
=
\frac{\mu(F)}{\mu(F')} k^{F}(B) + \frac{1}{\mu(F')} \int_{F' \setminus F } \! \langle \delta_0 , \chi_B (H_{\omega}) \delta_0\rangle \, d \mu(\omega).
\end{equation}

\begin{theorem}
Let $\Sigma$ denote the almost sure spectrum of the operators $H_{\omega}$ from Theorem \ref{t:as.spectrum}. One then has
\begin{equation} \label{eq:avronsimon.dos}
\overline{\bigcup_{F\in \mathcal{F}} \supp (d k^F)}
=
\Sigma.
\end{equation}
\end{theorem}

\begin{proof}

For notational ease, let $S$ denote the left hand side of \eqref{eq:avronsimon.dos}.  To prove the inclusion ``$\subseteq$,'' suppose that $E_0 \in \R \setminus \Sigma$.  We may then choose a continuous, nonnegative function $g$ for which $g(E_0) = 1$ and $g|_{\Sigma} \equiv 0$.  One then has $g(H_{\omega}) = 0$ for $\mu$ almost every $\omega \in \Omega$ by the spectral theorem.  From this, it follows that
$$
\int \! g(E) \, d k^F(E)
=
\frac{1}{\mu(F)} \int_{F} \! \langle \delta_0, g(H_{\omega})  \delta_0 \rangle \, d \mu(\omega) = 0
$$
for all $F \in \mathcal{F}$.  Thus, for all $F$, one has $E_0 \notin \supp (d k^F)$, so
$$
\bigcup_{F\in \mathcal{F}} \supp (dk^F)
\subseteq
\Sigma,
$$
which then implies $S \subseteq \Sigma$, since $\Sigma$ is closed. Conversely, given $E_0 \in \R \setminus S$, pick a continuous nonnegative function $g$ with $g(E_0) = 1$ such that $g$ vanishes on $S$.  For each $F \in \mathcal F$, we get
\begin{equation} \label{eq:spatial.dos.int}
0
=
\int \! g(E) \,  d k^F(E)
=
\frac{1}{\mu(F)} \int_{F} \! \langle \delta_0, g(H_{\omega}) \delta_0 \rangle \,  d \mu(\omega).
\end{equation}
By $\sigma$-finiteness, $\Omega$ enjoys a countable exhaustion $F_1 \subseteq F_2 \subseteq \cdots$ by members of $\mathcal{F}$.  Using \eqref{eq:spatial.dos.int}, for each $n$, $ \langle \delta_0, g(H_\omega) \delta_0 \rangle $ vanishes for $\mu$-almost every $\omega \in F_n$.  Thus, the same inner product vanishes for $\mu$-almost every $\omega \in \Omega$.  Lastly, note that
$$
\int_F \! \langle \delta_0, g(H_\omega) \delta_0 \rangle \, d\mu(\omega)
=
\int_F \! \langle S\delta_1, g(H_\omega) S\delta_1 \rangle \, d\mu(\omega)
=
\int_{T^{-1}(F)} \! \langle \delta_1, g(H_\omega) \delta_1 \rangle \, d\mu(\omega),
$$
so, by the same argument as before, $\langle \delta_1, g(H_\omega) \delta_1 \rangle =0$ for $\mu$-almost every $\omega \in \Omega$.  Since $\{\delta_0,\delta_1\}$ is a cyclic pair for $H_\omega$, $g(H_\omega) = 0$ for $\mu$ almost every $\omega \in \Omega$, which implies $E_0 \notin \Sigma$.
\end{proof}

Consider an exhaustion $F_1 \subseteq F_2 \subseteq \cdots$ of $\Omega$ by sets of finite measure, and abbreviate $dk^l =  dk^{F_l}$.  By general measure theory, there exists some weakly convergent subsequence $ dk^{l_j}$.  One may hope that the sequence $dk^l$ itself might be weakly convergent, but this is not the case.

\begin{theorem}
In general, the measures $ dk^l$ need not have a weak limit $ dk$ as $l\to \infty$.
\end{theorem}
\begin{proof}
To see this, it suffices to construct an example for which the first moments
$$
\int \! E \, dk^l(E)
$$
fail to converge as $l \to \infty$.

Take $\Omega = \R$ endowed with Lebesgue measure and an invertible, measure-preserving, ergodic, conservative transformation $T$, and put $F_l = [-l,l] $ for $l \in \Z_+$. Next, choose a sequence $(a_n) \in \{0,1\}^{\Z_+}$ such that the Cesar\`{o} averages
$$
s_l = \frac{a_1 + \cdots + a_l}{l}
$$
fail to converge as $l \to \infty$.  It is not hard to see that we may construct a bounded, continuous function $f$ so that
$$
\int_{-l-1}^{-l} \! f(x) \, dx
=
\int_l^{l+1} \! f(x) \, dx = a_l.
$$
With this setup, define the family of Schr\"odinger operators $(H_{x})_{x\in\R}$ as usual, and observe that
\begin{align*}
\int \! E \, dk^l(E)
& =
\frac{1}{\mu(F_l)} \int_{F_l} \! \langle \delta_0, H_{\omega}\delta_0 \rangle \,  d\mu(\omega) \\
& =
\frac{1}{2l} \int_{[-l,l]} \! f(x) \,  dx \\
& =
s_l,
\end{align*}
which fails to converge by construction. In particular, $dk^l$ is not weakly convergent.
\end{proof}

The example above generalizes readily.  Assume given $(\Omega, \mathcal B, \mu)$ which is $\sigma$-finite with $\mu(\Omega) = \infty$, and an exhaustion $F_1 \subseteq \cdots$ of $\Omega$ by sets of finite measure such that $\mu(F_n \backslash F_{n-1}) > 0 $ for $n>1$.  We can then choose a sequence $a_n$ so that the weighted Cesar\`{o} averages $s_l = \frac{a_1 + \cdots + a_l}{\mu(F_l)}$ fail to converge. With the convention $F_0 = \emptyset$, the choice
$$
f
=
\sum_{j=1}^\infty \frac{a_j}{\mu(F_j \setminus F_{j-1})} \chi_{F_j \setminus F_{j-1}}
$$
produces an example for which the spatial density of states cutoffs do not converge weakly.
\newline

\subsection{Thermodynamic Limit of Finite Truncations} In the finite-measure case, we can also view the density of states as a weak$^*$ limit of  averages of spectral measures or the weak$^*$ limit of uniform measures placed on the spectra of finite cutoffs.  More precisely, given $N \in \N$, let $P_{N,+}:\ell^2(\Z) \to \ell^2(\{0,\ldots, N-1\})$ and $P_{N,-}:\ell^2(\Z) \to \ell^2(\{-N,\ldots,-1\})$ denote the canonical projections. Then, for $\omega \in \Omega$ and $N \in \Z_+$, define probability measures $dk_{\omega,N}^\pm$ and $d\tilde{k}_{\omega,N}^\pm$ on $\R$ via
 \begin{align*}
 \int \! g(E) \, dk_{\omega,N}^\pm (E)
 & =
 \frac{1}{N} \tr\left(P_{N,\pm} g(H_\omega) P_{N,\pm}^*\right) \\
  \int \! g(E) \, d\tilde k_{\omega,N}^\pm (E)
 & =
 \frac{1}{N} \tr\left( g\left(P_{N,\pm} H_\omega P_{N,\pm}^* \right) \right)
 \end{align*}
for each continuous function $g$. For later use, we point out that taking $g(E) \equiv E$, we have
\begin{align} \label{eq:dosRight}
 \int E \, dk_{\omega,N}^+ (E)
=
 \int E \, d\tilde k_{\omega,N}^+ (E)
 &=
 \frac{1}{N} \sum_{n=0}^{N-1} f(T^n\omega) \\[1mm]
 \nonumber & = A_N(\omega,f,T), \\[2mm]
 \label{eq:dosLeft}
  \int E \, dk_{\omega,N}^- (E)
=
 \int E \, d\tilde k_{\omega,N}^- (E)
 & =
 \frac{1}{N} \sum_{n=-N}^{-1} f(T^n\omega) \\[1mm]
 \nonumber
 &  = A_N(\omega,f \circ T^{-1},T^{-1}).
\end{align}

\begin{theorem} \label{t:dos.I}
There exists $\Omega_* \subseteq \Omega$ of full $\mu$-measure such that, for every continuous function $g$, there exist constants $\overline{I}^\pm(g) $ and $\underline{I}^\pm(g)$ such that
\begin{align}
\label{eq:dos.I1}
\underline{I}^\pm(g)
& =
\liminf_{N \to \infty} \int \! g \,  dk_{\omega,N}^\pm \\
\label{eq:dos.I1a}
 & =
\liminf_{N\to\infty} \int \! g \, d\tilde k_{\omega,N}^\pm \\
\label{eq:dos.I2}
\overline{I}^\pm(g)
& =
\limsup_{N \to \infty} \int \! g \,  dk_{\omega,N}^\pm \\
\label{eq:dos.I2a}
& =
\limsup_{N\to\infty} \int \! g \, d\tilde k_{\omega,N}^\pm.
\end{align}
for all $\omega \in \Omega_*$.
\end{theorem}

\begin{proof}
First, notice that
\begin{align*}
\int \! g(E) \,  dk_{\omega,N+1}^+(E)  & = \frac{1}{N+1} \sum_{j=0}^{N} \langle \delta_j, g(H_{\omega}) \delta_j \rangle  \\
                              & = \frac{N}{N+1} \frac{1}{N}   \sum_{j=1}^{N} \langle \delta_j, g(H_{\omega}) \delta_j \rangle + \frac{1}{N+1} \langle \delta_0, g(H_{\omega}) \delta_0 \rangle \\
                              & = \frac{N}{N+1} \frac{1}{N}   \sum_{j=0}^{N-1} \langle \delta_j, g(H_{T\omega}) \delta_j \rangle + \frac{1}{N+1} \langle \delta_0, g(H_{\omega}) \delta_0 \rangle \\
                              & = \frac{N}{N+1} \int \! g(E) \, dk^+_{T\omega,N}(E) +\frac{1}{N+1} \langle \delta_0, g(H_{\omega}) \delta_0 \rangle.
\end{align*}
Taking lim inf and lim sup of both sides proves that $\liminf \int g \, dk_{\omega,N}^+$ and $\limsup \int g \, dk_{\omega,N}^+$ are $T$-invariant functions of $\omega$. A similar argument shows that this holds with $-$ replacing $+$. Thus, we find a full-measure set $\Omega_g$ and constants $\underline{I}^\pm(g)$, $\overline{I}^\pm(g)$ so that  \eqref{eq:dos.I1} and \eqref{eq:dos.I2} hold true for $\omega \in \Omega_g$.

Next, let $\mathcal P$ denote the collection of all polynomials having rational coefficients, and $\Omega_* = \bigcap_{p \in \mathcal P} \Omega_p$. Uniformly approximating a continuous function $g$ by $p \in \mathcal P$ on $K:=[-2-\|f\|_\infty,2+\|f\|_\infty]$, we observe \eqref{eq:dos.I1} and \eqref{eq:dos.I2} hold for all continuous $g$ and all $\omega \in \Omega_*$.

Next, consider $p \in \mathcal P$ and $\omega \in \Omega_*$. By an explicit calculation, one has
\[
\left| \int \! p \, dk_{\omega,N}^\pm - \int \! p\, d\tilde k_{\omega,N}^\pm \right|
=
O(1/N),
\]
where the implicit constant depends on $p$ but not on $N$. Thus, \eqref{eq:dos.I1a} and \eqref{eq:dos.I2a} hold for $p \in \mathcal P$. Passing to general $g$ via uniform approximation concludes the proof.
\end{proof}

Theorem~\ref{thm:0} shows us that Theorem \ref{t:dos.I} is optimal in the sense that we cannot expect the ``upper'' and ``lower'' density of states limits to agree.  To see this, assume given $(\Omega, \mathcal B, \mu, T)$ ergodic, conservative, invertible, and $\sigma$-finite with $\mu(\Omega) = \infty$, and choose a measurable function $f : \Omega \to \{0,1\}$ as in Theorem~\ref{thm:0}.  Define $V_{\omega}(n) = f(T^n \omega)$ and $H_{\omega} = \Delta + V_{\omega}$ as usual.  Then, with $g(E) \equiv E$, Theorem~\ref{thm:0} and \eqref{eq:dosRight} imply
\begin{align*}
\underline I^+ (g) = 0 & \neq 1 = \overline I^+(g).
\end{align*}

Additionally, choosing $f$ as in Section~\ref{sec:ex} \eqref{eq:an1} and \eqref{eq:an0} imply
\[
\overline{I}^-(g) = \underline{I}^-(g) = 0 \neq 1 =\overline{I}^+(g),
\]
so the behavior on the left and right half-lines may not be the same.

\section{The Lyapunov Exponents}\label{sec.le}

As before, let $(\Omega,\mathcal B,\mu)$ be a measure space, and $T$ a non-singular invertible ergodic map. Throughout this section, we will also assume that $T$ is conservative.

Assume that $f$ and $H_\omega$ are defined as above. Let us define the one-step transfer matrix
\[
A(E,\omega) =  \begin{pmatrix} E - f(\omega) & -1 \\ 1 & 0 \end{pmatrix}
\]
and the $n$-step transfer matrix
\[
A(E,n,\omega) = \begin{cases} A(E,T^{n-1}\omega) \dotsm  A(E, \omega) & n  > 0, \\
I & n =0,\\
A(E,T^{-\lvert n\rvert}\omega)^{-1} \dotsm A(E,T^{-1}\omega)^{-1} & n < 0.
\end{cases}
\]
In this section we explore the growth of norms of transfer matrices and its relation to the spectral properties of $H_\omega$.

\begin{theorem}
For any $E\in \mathbb{C}$, there exist finite numbers $\overline L^\pm(E)$, $\underline L^\pm(E)$, called upper and lower Lyapunov exponents, such that for $\mu$-a.e.\ $\omega$,
\begin{align}
 \limsup_{n\to\pm \infty} \frac 1{\lvert n\rvert} \log \lVert A(E,n,\omega) \rVert  & = \overline L^\pm(E) \\
 \liminf_{n\to\pm \infty}  \frac 1{\lvert n\rvert} \log \lVert A(E,n,\omega) \rVert & = \underline L^\pm(E).
\end{align}
\end{theorem}

\begin{proof}
Denote
\begin{equation}
f_n(\omega) = \log \lVert A(E,n,\omega) \rVert
\end{equation}
and
\begin{align}
\overline f^\pm(\omega) & = \limsup_{n\to \pm\infty} \frac 1{\lvert n\rvert} f_n(\omega) \\
\underline f^\pm(\omega) & = \liminf_{n\to \pm\infty} \frac 1{\lvert n\rvert}  f_n(\omega).
\end{align}
We prove that $\overline f^\pm(\omega)$ and $\underline f^\pm(\omega)$ are $\mu$-a.e.\ constant.

Sub-multiplicativity of the matrix norm implies that for $m, n \ge 0$,
\[
f_{m+n}(\omega) \le f_m(\omega) + f_n (T^m \omega).
\]
In particular, with $m=1$, this implies
\[
\frac{f_{n+1}(\omega)}n \le \frac{f_1(\omega)}n + \frac{f_n(T\omega)}n
\]
and taking the $\liminf$ as $n\to +\infty$, we conclude
\[
\underline f^+(\omega) \le \underline f^+(T\omega)
\]
Thus, for any $\gamma\in\mathbb{R}$, the set
\[
B_\gamma = \{ \omega : \underline f^+(\omega) < \gamma \}
\]
obeys $T^{-1} B_\gamma \subset B_\gamma$. Since $T$ is conservative and ergodic, this implies that $\mu(B_\gamma)=0$ or $\mu(B_\gamma^c)=0$. 
Thus, there is a constant $c$ such that $\underline f^+(\omega)=c$ for $\mu$-a.e.\ $\omega$. This constant is precisely $\underline L^+(E)$.

The proof for the other three constants is analogous.
\end{proof}

Obviously, $\underline L^+(E) \le \overline L^+(E)$ and $\underline L^-(E) \le \overline L^-(E)$. But there are also inequalities between Lyapunov exponents at $+\infty$ and those at $-\infty$.

\begin{prop}
Both lower Lyapunov exponents are smaller or equal than both upper Lyapunov exponents, i.e.\
\begin{align*}
\underline L^-(E) \le \overline L^+(E), \qquad \underline L^+(E) \le \overline L^-(E).
\end{align*}
\end{prop}

\begin{proof}
Use the same notation as in the proof of the previous theorem. Notice that for $n>0$,
\begin{equation}\label{ml01}
f_{-n}(\omega) = f_n(T^{-n}\omega).
\end{equation}

Let $\delta < \underline L^+(E)$. Then for $\mu$-a.e.\ every $\omega$, the inequality $f_n(\omega)/n \le \delta$ holds for finitely many positive values of $n$. Thus, denoting
\[
A_n = \{ \omega : f_n(\omega)/n > \delta \}
\]
we have
\[
\mu (\Omega \setminus \cup_{m\ge 1} \cap_{n\ge m} A_n ) = 0
\]
so for some value of $m\ge 1$, the set $W = \cap_{n\ge m} A_n$ obeys
\[
\mu ( W ) > 0.
\]
Since $T$ is invertible and ergodic, by \cite[Prop.~1.2.2]{A},
for $\mu$-a.e.\ $\omega$, $\omega$ is in $T^{n} W$ for infinitely many values of $n$, so by \eqref{ml01}, $f_{-n}(\omega)/n > \delta$ for infinitely many values of $n$. Thus, $\overline L^-(E) > \delta$.

Since this holds for any $\delta< \underline L^+(E)$, we have shown $\overline L^-(E) \ge \underline L^+(E)$. The other inequality is analogous.
\end{proof}

However, the upper and lower Lyapunov exponents are not necessarily equal. To see this, we will rely on the construction in Section~\ref{SBirkhoff} and the avalanche principle. The avalanche principle was introduced by Goldstein--Schlag~\cite{GS}; we will use a strengthened version due to Bourgain--Jitomirskaya~\cite{BJ}.

\begin{lemma}[{\cite[Lemma 5]{BJ}}] \label{L4.4} Let $\mu$ be sufficiently large, $N=3^s$, and $A_1,\dots,A_N \in \SL(2,\mathbb{R})$ such that $\lVert A_j \rVert \ge \mu$ and
\[
\left\lvert \log \lVert A_j\rVert + \log \lVert A_{j+1} \rVert - \log\lVert A_{j+1} A_j \rVert \right\rvert  < \tfrac 12 \log \mu.
\]
Then
\[
\left\lvert \log \left\lVert \prod_{j=N}^1 A_j \right\rVert  + \sum_{j=2}^{N-1} \log \lVert A_j \rVert - \sum_{j=2}^{N-1} \log\lVert A_{j+1} A_j \rVert \right\rvert  <  C_1 \frac N \mu,
\]
where $C_1$ is an absolute constant.
\end{lemma}

Note that if all the all conditions of the above lemma hold and
\[
\left\lvert \log \lVert A_j\rVert + \log \lVert A_{j+1} \rVert - \log\lVert A_{j+1} A_j \rVert \right\rvert  < \gamma
\]
for some $\gamma\le \tfrac 12 \log \mu$, then the above inequalities imply
\begin{equation}\label{ml60}
\log \left\lVert \prod_{j=N}^1 A_j \right\rVert  \ge \sum_{j=2}^{N-1} \log \lVert A_{j+1} \rVert - (N-2) \gamma - C_1 \frac{N}\mu,
\end{equation}
which is the form we will use below.

\begin{prop}
For large enough $M>0$, there exist bounded sampling functions $f:\Omega \to \mathbb{R}$ such that $\overline{L}^+(E) > \underline{L}^+(E)$ when $\lvert E \rvert > M$.
\end{prop}

\begin{proof} Let us follow the construction in Theorem~\ref{thm:0}, noting that we can force all the numbers $N_k$ in that construction to be powers of $3$. We pick a sampling function $f$ such that the potential takes two possible values, $v_1$ and $v_2$, and that $\mu$-almost surely,
\begin{align*}
\limsup_{s\to\infty} \frac 1{3^s} \lvert \{ j\in\mathbb{Z} : 1\le j \le 3^s, f(T^j \omega) = v_1 \} \rvert & = 1 \\
\liminf_{s\to\infty} \frac 1{3^s} \lvert \{ j\in\mathbb{Z} : 1\le j \le 3^s, f(T^j \omega) = v_1 \} \rvert & = 0
\end{align*}
Then, for every $E$,  $\omega$-almost surely, denoting $A(x) = \begin{pmatrix} x & -1 \\ 1 & 0 \end{pmatrix}$,
\begin{align*}
\limsup_{s\to\infty} \frac 1{3^s} \sum_{j=1}^{3^s} \log \lVert A(E,T^j\omega) \rVert  & = \max_i \log \lVert A(E-v_i) \rVert \\
\liminf_{s\to\infty} \frac 1{3^s} \sum_{j=1}^{3^s} \log \lVert A(E,T^j\omega) \rVert & = \min_i \log \lVert A(E-v_i) \rVert
\end{align*}
Sub-multiplicativity of matrix norms guarantees that
\[
\ul L^+(E) \le \liminf_{s\to\infty} \frac 1{3^s} \log \left\lVert  A(E,3^s,\omega) \right\rVert \le  \min_i \log \lVert A(E-v_i) \rVert.
\]
If the avalanche principle is applicable to $A(E,j,\omega)$ with a suitable choice of $\mu$, \eqref{ml60} implies that
\[
\ol L^+(E) \ge \limsup_{s\to\infty} \frac 1{3^s} \log \left\lVert  A(E,3^s,\omega) \right\rVert \ge \max_i \log \lVert A(E-v_i) \rVert  - \frac{C_1}\mu - \gamma.
\]
Thus, $\overline L^+(E) > \underline L^+(E)$ will follow from
\begin{equation}\label{ml65}
 \max_i \log \lVert A(E-v_i) \rVert  -  \min_i \log \lVert A(E-v_i) \rVert  > \frac{C_1}\mu + \gamma.
\end{equation}

Thus, it suffices to show that there is a suitable choice of $v_1, v_2$ such that, for all large enough $E$, there are choices of $\mu$, $\gamma$ such that the avalanche principle is applicable and \eqref{ml65} holds. We will now show that this is true if we choose
\[
v_1 = \delta = 2C_1, \qquad v_2 = - \delta, \qquad \mu = E-\delta, \qquad \gamma = \frac 4{(E-\delta)^2}.
\]
For large enough $E$, it is then obvious that
\[
\gamma < \frac 12 \log \mu, \quad \frac{\delta}{1+E+\delta} > \frac {C_1}\mu + \gamma.
\]
We will now need some norm estimates. Let
\begin{align*}
g(x) & = \frac 12 \log\left( 1+ \frac{x^2}2 + \sqrt{x^2 + \frac{x^4}4} \right).
\end{align*}
If $A \in \SL(2,\mathbb{R})$ and $\Tr(A^*A)=2+x^2$ for some $x\in\mathbb{R}$, then
\begin{align*}
\log \lVert A\rVert & = g(x),
\end{align*}
since $\lVert A\rVert^2$ is the larger eigenvalue of $A^*A$ and eigenvalues of $A^*A$ are the solutions of $\lambda^2 - (2+x^2) \lambda +1=0$. It is straightforward to compute $\Tr(A(x)^* A(x))$ and  $\Tr(A(y)^* A(x)^* A(x) A(y))$ to see
\begin{align*}
\log \lVert A(x) \rVert & = g(x), \\
\log \lVert A(x)A(y) \rVert &= g(\sqrt{x^2y^2 + (x-y)^2}).
\end{align*}
For $x>0$,
\[
g(x) = \frac 12 \log\left( 1+ \frac{x^2}2 + \sqrt{x^2 + \frac{x^4}4} \right) \ge \frac 12 \log\left( 1+ \frac{x^2}2 + \frac{x^2}2 \right) \ge \frac 12 \log(x^2) = \log x,
\]
which implies that
\[
\log\lVert A(E\pm \delta) \rVert \ge \log \mu.
\]
In the opposite direction, for $x>0$, we use $\sqrt{1+ 4/x^2} \le 1 + 2/x^2$ to estimate
\begin{equation*}
g(x) - \log x
= \frac 12 \log\left( \frac 1{x^2}+ \frac 12 + \frac 12 \sqrt{1 + \frac 4{x^2}} \right)
  \le  \frac 12 \log\left( 1 + \frac 2{x^2}  \right)  \le \frac 1{x^2}.
\end{equation*}

For $1 \le x\le y$, using this inequality three times and noting $\sqrt{x^2 y^2 + (x-y)^2}\ge x$, we get
\begin{align*}
\left\lvert g(x) + g(y) - g(\sqrt{x^2y^2 + (x-y)^2}) \right\rvert  &  \le \frac 3{x^2}  + \frac{1}{2} \log \frac {x^2y^2 + (x-y)^2}{x^2 y^2}  \le \frac 4{x^2} \label{ml61}
\end{align*}
where, for the last step, we used $\log (1 + \frac {(x-y)^2}{x^2 y^2}) \le \frac {(x-y)^2}{x^2 y^2} \le \frac 1{x^2}$.
Thus, for large enough $E$ and $x, y \in \{ E-\delta, E+\delta\}$,
\[
\left\lvert \log \lVert A(x)\rVert + \log \lVert A(y) \rVert - \log\lVert A(x) A(y) \rVert \right\rvert  < \gamma.
\]

For $x>0$,
\begin{align*}
g'(x) & = \frac 12 \frac{x+ \frac{2x+x^3}{2\sqrt{x^2+\frac{x^4}4}}}{1+\frac{x^2}2 + \sqrt{x^2+\frac{x^4}4}}  \ge \frac 12 \frac{x+ 1}{1+\frac{x^2}2 + x + \frac{x^2}2}  \ge \frac 1{2(1+x)}
\end{align*}
so, by the mean value theorem,
\[
 g(x+\delta) - g(x-\delta)  \ge \frac {\delta}{1+x+\delta}
\]
for $x>\delta$. In particular,
\[
g(E+\delta) - g(E-\delta) > \frac{C_1}\mu + \gamma.
\]
By these estimates, the avalanche principle is applicable and, by the estimates above, $\overline L^+(E) > \underline L^-(E)$ for all large enough $E$. Note that $g$ is an even function so the above discussion applies with minimal modifications to the case of negative $E$ with large enough $\lvert E\rvert$. This completes the proof.
\end{proof}

\begin{prop}
There exists a conservative, invertible, measure preserving ergodic
transformation $(\Omega,\B,\mu,T)$, a sampling function $f\colon \Omega\to \mathbb{R}$
and $M>0$ such that
\[
\ol{L}^+(E) > \ol{L}^-(E)
\]
when $E > M$, and
\[
\ol{L}^+(E) < \ol{L}^-(E)
\]
when $E < - M$.
\end{prop}

\begin{proof}
We start with the ergodic system and function $f$ constructed in Section~\ref{sec:ex}. We rescale $f$ so that it takes two possible values, $v_1 = \delta$ and $v_2 = -\delta$, such that for almost every $\omega$,
\begin{align*}
\limsup_{N\to\infty} \frac 1{N} \lvert \{ j\in\mathbb{Z} : 1\le j \le N, f(T^j \omega) = \delta \} \rvert & = 1, \\
\lim_{N\to\infty} \frac 1{N} \lvert \{ j\in\mathbb{Z} : 1\le j \le N, f(T^{-j} \omega) = \delta \} \rvert & = 0.
\end{align*}
The first of these inequalities implies
\begin{equation}\label{ml99}
\limsup_{s\to\infty} \frac 1{3^s} \lvert \{ j\in\mathbb{Z} : 1\le j \le 3^s, f(T^j \omega) = \delta \} \rvert \ge \frac 13.
\end{equation}
From here, we use the same approach as in the previous proof: the avalanche principle is used to prove that different asymptotics of Birkhoff averages imply different asymptotics of the subadditive logs of matrix norms. If we choose
\[
\delta = 6 C_1, \qquad \mu = E-\delta, \qquad \gamma = \frac 4{(E-\delta)^2},
\]
(the extra factor of $3$ for $\delta$ comes from the factor of $3$ in \eqref{ml99}), since then
\[
\gamma < \frac 12 \log \mu,\qquad \frac{\delta}{1+E+\delta} > 3 \left( \frac{C_1} \mu + \gamma \right)
\]
and we prove as in the previous proof that
\[
\ol L^+(E) \ge \frac 13 g(E+\delta) + \frac 23 g(E-\delta) - \frac{C_1}\mu - \gamma  >  g(E-\delta) \ge \ol L^-(E). \qedhere
\]
\end{proof}

An analogous argument proves the analogous proposition for lower Lyapunov exponents:

\begin{prop}
There exists a conservative, invertible, measure preserving ergodic
transformation $(\Omega,\B,\mu,T)$, a sampling function $f\colon \Omega\to \mathbb{R}$
and $M>0$ such that
\[
\ul{L}^+(E) > \ul{L}^-(E)
\]
when $E > M$, and
\[
\ul{L}^+(E) < \ul{L}^-(E)
\]
when $E < - M$.
\end{prop}

The following is the extension of the Ishii--Pashtur theorem to the infinite measure setting.

\begin{theorem}
\[
\Sigma_\ac \subset \overline{\{ E\in\mathbb{R} : \overline L^+(E) = 0\text{ or }\overline L^-(E)=0 \}}^\ess
\]
\end{theorem}

\begin{proof}
Denote
\[
\mathcal Z = \{ E\in\mathbb{R} : \overline L^+(E) = 0\text{ or }\overline L^-(E)=0 \}.
\]
For every $E\in\mathbb R \setminus \mathcal Z$,
\begin{equation}\label{ml02}
\ol f^-(E,\omega)>0\text{ and }\ol f^+(E,\omega)>0
\end{equation}
holds for a.e. $\omega\in \Omega$. Since $f_n(E,\omega)$ are measurable functions, so are $\ul f^\pm(E,\omega)$ and $\ol f^\pm(E,\omega)$; thus, the set
\[
\{ (E,\omega) \in \mathbb R \times \Omega : \text{\eqref{ml02} holds} \}
\]
is measurable. Thus, by Fubini's theorem, for $\mu$-a.e.\ $\omega\in \Omega$, there is a set $B_\omega$ with $\lvert B_\omega\rvert =0$ such that \eqref{ml02} holds for all $E \in (\mathbb{R} \setminus \mathcal Z) \setminus B_\omega$.

By a result of Last--Simon~\cite[Theorem 3.10]{LS}, for a.e.\ $E$ w.r.t.\ the absolutely continuous part of the spectral measure of $H_\omega$, we have for at least one choice of the $\pm$ sign,
\begin{equation}\label{ml10}
\limsup_{N\to\infty} \frac 1{N \log^2 N} \sum_{n=1}^N \lVert A(E, \pm n,\omega) \rVert^2  < \infty
\end{equation}
(the theorem of Last--Simon is stated for half-line operators, but that implies the whole line result using standard arguments).

However, it is easy to see that $\ol f^\pm(E,\omega)>0$ implies that the corresponding $\limsup$ in \eqref{ml10} is $+\infty$. Thus,  $\mathcal P_\omega^{(\ac)}((\mathbb{R} \setminus \mathcal Z) \setminus B_\omega) =0$, and $\lvert B_\omega\rvert =0$ then implies
\[
\mathcal P_\omega^{(\ac)}(\mathbb{R} \setminus \mathcal Z) =0
\]
for $\mu$-a.e.\ $\omega$. Thus, $\sigma_\ac(H_\omega) \subset \ol{\mathcal Z}^\ess$ for $\mu$-a.e.\ $\omega$, which completes the proof.
\end{proof}

Conspicuously absent from our discussion here is a version of Kotani theory (see, e.g., \cite{D07, K84, S83} for some papers on Kotani theory in the finite measure case) or some hints as to why no natural analogue exists. We regard results in this direction for the infinite measure case as very interesting.

\end{document}